\documentclass[11pt]{amsart}
\usepackage[margin=1.5cm]{geometry}
\usepackage{mathrsfs}
\usepackage{epic,eepic,epsf,epsfig,tikz}
\usetikzlibrary{positioning}
\usepackage{amsmath}
\usepackage{amssymb}
\usepackage{amsbsy}
\usepackage{amsfonts}
\usepackage{color}

\numberwithin{equation}{section}
\newtheorem{lem}{Lemma}[section]%
\newtheorem{theorem}[lem]{Theorem}%
\newtheorem{cor}[lem]{Corollary}%
\newtheorem{problem}[lem]{Problem}%
\newtheorem{prop}[lem]{Proposition}%

  \def\G{\Gamma}

\def\nd{\mathrel{\bigm|\kern-.7em/}}

\def\f{\noindent}

\def\Aut{\hbox{\rm Aut}}
\def\Cay{\hbox{\rm Cay}}

\def\mz{{\mathbb Z}}

\begin{document}
\title[]{On oriented $m$-semiregular representations of finite groups about valency two}

\author{Jia-Li Du}
\address{Jia-Li Du, School of Mathematics, China University of Mining and Technology, Xuzhou 221116, China}
\email{dujl@cumt.edu.cn}

\author{Young Soo Kwon}
\address{Young Soo Kwon, Mathematics, Yeungnam University, Kyongsan 712-749, Republic of Korea}
\email{ysookwon@ynu.ac.kr}

\author{Da-Wei Yang$^*$}
\address{Da-Wei Yang, School of Science, Beijing University of Posts and Telecommunications, Beijing 100876, China}
\email{dwyang@bupt.edu.cn}

\date{}
 \maketitle

\begin{abstract}
Given a group $G$, an {\em $m$-Cayley digraph $\G$ over $G$} is a digraph that has a group of automorphisms isomorphic to $G$ acting semiregularly on the vertex set with $m$ orbits. We say that $G$ admits an {\em oriented $m$-semiregular representation} (O$m$SR for short), if there exists a regular $m$-Cayley digraph $\G$ over $G$ such that $\G$ is oriented and its automorphism group is isomorphic to $G$. In particular, O$1$SR is also named as ORR. Verret and Xia gave a classification of finite simple groups admitting an ORR of valency two in [Ars Math. Contemp. 22 (2022), \#P1.07].
Let $m\geq 2$ be an integer. In this paper, we show that all finite groups generated by at most two elements admit an O$m$SR of valency two except four groups of small orders. Consequently, a classification of finite simple groups admitting an O$m$SR of valency two is obtained.

\bigskip

\f {\bf Keywords:} Semiregular group, regular representation, $m$-Cayley digraph, oriented $m$-semiregular representation.

\medskip
\f {\bf 2010 Mathematics Subject Classification:} 05C25, 20B25.
\end{abstract}

\section{Introduction}

By a {\em digraph} $\Gamma$, we mean an ordered pair $(V, A)$ where the vertex set
$V$ is a non-empty set and the arc set $A$ is a binary relation on $V$, that is,
$A \subseteq V \times V$. The elements of $V$ and $A$ are called vertices and arcs
of $\Gamma$, respectively. For simplicity, we write $V(\Gamma):=V$ and $A(\Gamma):=A$.
The digraph $\Gamma $ is
a {\em graph} if the binary relation $A$ is symmetric, that is, $A=\{(v,u)\ |\ (u,v)\in A\}$.
A digraph is called {\em regular} if each vertex has the same in- and out-valency. For a vertex $u \in V(\Gamma)$, let $\Gamma^{+}(u)$ be the set of vertices $v$ such that $(u,v) \in A(\Gamma)$.
Throughout this paper, all digraphs are regular, and all digraphs and groups are finite.

Let $\Gamma$ be a digraph, and let $\omega \in V(\Gamma)$.  An automorphism of $\Gamma$ is a permutation $\sigma$ of $V(\Gamma)$ fixing $A(\Gamma)$ setwise, that is,
$(x^\sigma , y^\sigma ) \in A(\Gamma)$ if and only if $(x, y) \in A(\Gamma)$. The set of all automorphisms of $\Gamma$, with the operation of composition, is  {\em the full automorphism group} of $\G$, denoted by $\Aut(\Gamma)$. For a subgroup $G$ of $\Aut(\Gamma)$,  denote by $G_\omega$ the stabilizer of $\omega$ in $G$, that is, the subgroup of $G$ fixing $\omega$. We say that $G$ is {\em semiregular} on $V(\Gamma)$ if $G_\omega = 1$ for every $\omega \in V(\Gamma)$, and {\em regular} if it is semiregular and transitive.

Let $G$ be a group, and let $S$ be a subset of $G$.
The {\em Cayley digraph} $\G:=\Cay(G,S)$ is the digraph with $V(\Gamma):=G$ and
$A(\G):=\{(g,sg) ~|\ g\in G,s\in S\}$. In particular, $\G$ is a Cayley graph if
and only if $S=S^{-1}$. The right regular representation $R(G)$ of $G$ gives rise to an
embedding of $G$ into $\Aut(\Gamma)$, that is, $R(G)$ is a subgroup of $\Aut(\Gamma)$.
If $\Aut(\Gamma)$ actually coincides with $R(G)$, then the (di)graph $\Gamma$ is
said to be a {\em {\rm(}di{\rm)}graphical
regular representation over $G$}, and this is usually abbreviated to GRR (or DRR). GRRs (or DRRs) offer a natural way to represent groups geometrically and combinatorially as groups of automorphisms of (di)graphs~\cite{MorrisSpiga1}.
The problem that determining the finite groups admitting a GRR or DRR has a long history. The complete classification of finite groups admiting a DRR has been given by Babai~\cite{Babai}, and it was shown that except for five small groups, every group admits a DRR. Clearly, if
a group $G$ admits a GRR, then $G$ admits a DRR, however the converse is not true. The GRR turned out to be much more difficult to handle and, after a long series of partial results by various authors~\cite{Hetzel,Imrich,ImrichWatkins,ImrichWatkins2,NowitzWatkins1,NowitzWatkins2,Watkins}, the classification was completed by Godsil in~\cite{Godsil}.
Once the classifications of DRRs and GRRs were completed, researchers proposed and investigated various natural generalizations. In the following paragraphs, we mainly introduce some background about oriented $m$-semiregular representation; one may see~\cite{MorrisSpiga1,MorrisSpiga3,Spiga} for the introduction of some other generalizations.

{\bf Oriented regular representation, i.e. ORR}. A digraph is {\em oriented} if it does not have both oppositely directed arcs between any pairs of vertices. An \emph{oriented Cayley digraph} is in some sense a proper digraph.
More formally, it is a Cayley digraph $\Cay(G, S)$ whose connection set $S$ has the
property that $S \cap S^{-1} =\emptyset$. Equivalently, in graph-theoretic terms,
it is a digraph with no digons, see~\cite{MorrisSpiga1}. We say that a group $G$ admits an {\em oriented regular representation} (ORR for short) if there exists an oriented Cayley digraph $\G=\Cay(G,S)$ over $G$ such that $\Aut(\G)=R(G)$. In \cite[Problem 2.7]{Babai}, Babai asked which groups admit an ORR, and this problem was finally settled by Morris and Spiga in three papers~\cite{MorrisSpiga1,MorrisSpiga3,Spiga}.

{\bf Oriented $m$-semiregular representation, i.e. O$m$SR}.  The concept of Cayley digraphs can be nicely generalized to $m$-Cayley digraphs where regular actions are replaced with semiregular actions.
An $m$-Cayley (di)graph $\Gamma$ over a finite group $G$ is defined as a (di)graph which has a semiregular group of automorphisms
isomorphic to $G$ with $m$ orbits on its vertex set. Actually, $1$-Cayley (di)graphs are the usual Cayley (di)graphs and, $2$-Cayley (di)graphs
are also called bi-Cayley (di)graphs, see \cite{Kovacs}.
We say that a group $G$ admits a {\em $($di$)$graphical $m$-semiregular representation} (G$m$SR and D$m$SR, for short),
if there exists a regular $m$-Cayley (di)graph $\Gamma$ over $G$ such that $\Aut(\Gamma)=R(G)$, and admits an {\em oriented $m$-semiregular representation}
($\rm O$$m$$\rm SR$ for short), if there exists a regular $m$-Cayley digraph $\G$ over $G$ such that it is oriented and $\Aut(\G)=R(G)$.
In particular, G$1$SRs and D$1$SRs are the usual GRRs and DRRs, respectively.
The groups admitting G$m$SRs, D$m$SRs or O$m$SRs for every positive integer $m$ were determined in~\cite{DFS,DuFengSpiga}.

{\bf Oriented $m$-semiregular representation of prescribed valency}.
In contrast to unrestricted G$m$SRs, D$m$SRs and O$m$SRs, the classification of groups admitting G$m$SRs, D$m$SRs, or O$m$SRs of prescribed valency is largely open.

Since a connected graph of valency one or two is an edge or a cycle, the smallest interesting case is of valency three. The study on cubic GRR has begun with
a piece of Godsil's work~\cite{G}, which confirmed the existence of cubic GRR of the symmetric and alternating group. Recently, lots of research has been done on finite simple groups admitting a cubic GRR.  For example, Xia and Fang~\cite{XiaF,  Xia1,Xia}  studied some simple groups of Lie type admitting a cubic GRR. Moreover, Spiga~\cite{Spiga2} gave some sufficient conditions for a cubic Cayley graph of a non-abelian simple group to be a GRR.

For digraphs, the smallest interesting case is of valency two, as a connected digraph of valency one is just a directed cycle. Recently, Verret and Xia~\cite{VerretXia} classified simple groups admitting an ORR (i.e. O1RR) of valency two; every simple group of order at least 5 has an ORR of valency two.
This work prompts us to further consider the simple groups admitting an O$m$SR of valency two for $m\geq2$. Due to the classification of finite simple groups, we know that every finite simple
group can be generated by two elements, one can see~\cite[Corollary]{GKantor}.
In this paper, we consider O$m$SRs of valency two for groups generated
by at most two elements, where $m\geq2$.
Our main results are the following theorem.

%

\begin{theorem}\label{theo=main2}
Let $G$ be a finite group generated by at most two elements, and let $m\geq 2$ be an integer. Then $G$ admits no O$m$SR of valency two if and only if one of the following occurs:
\begin{enumerate}
\item\label{part1} $m=2$ and $G\cong \mz_1$, $\mz_2$ or $\mz_2^2$;
\item\label{part2} $3\leq m\leq 6$ and $G\cong \mz_1$.
\end{enumerate}

\end{theorem}


%

Consequently, a complete classification of simple groups admitting an O$m$SR of valency two can be obtained.

\begin{cor}\label{theo=main4}
Let $G$ be a finite simple group, and let $m\geq2$ be an integer.
Then $G$ admits an O$m$SR of valency two except $(m,G)=(2,\mz_2)$.
\end{cor}

We can see from Theorem~\ref{theo=main2} that almost all finite groups generated by two elements have an O$m$SR of valency two when $m\geq2$.
However, this is not true when $m=1$. A counterexample is the dihedral group, one may also see~\cite{Babai}. Moreover, the abelian group $\mz_n\times \mz_n$ also has no ORR of valency two, since it has an automorphism interchanging the two generators for each minimal generating set. To end this section, we would like to propose the following problem.

\begin{problem}\label{question}
Classify finite groups generated by at most two elements admitting an ORR of valency two.
\end{problem}

\section{Preliminaries and notations}

Let $m$ be a positive integer, and let $G$ be a group. Consistently throughout the whole paper, for not making our notation too cumbersome to use, we denote the element $(g,i)$ of the cartesian product $G\times\{0,\ldots,m-1\}$ simply by $g_i$. We often identify $\{0,\ldots,m-1\}$ with $\mathbb{Z}_m$, that is, with the integers modulo $m$.

For every $i,j\in\mathbb{Z}_m$, let $T_{i,j}$ be a subset of $G$. The {{\em $m$-Cayley digraph}} of $G$ with respect to $(T_{i,j}: i,j\in\mz_m)$ is the digraph with vertex set $$G\times \mathbb{Z}_m=\bigcup_{i\in\mz_m} G_i,$$ where $G_i=\{g_i\ |\ g\in G\}$, and with arc set
 $$\bigcup_{i,j} \{(g_i, (tg)_j)~|~t\in T_{i,j},g\in G\}.$$
We denote this digraph by
$$\Cay(G,T_{i,j}: i,j\in\mz_m).$$
For any given $g\in G$, the right multiplication $R(g)$ mapping each vertex $x_i\in G_i$ to $(xg)_i\in G_i$ for all $i\in\mz_m$, is an automorphism of $\G$, and $R(G)=\{ R(g)\ |\ g\in G\}$ is a semiregular group of automorphisms of $G$ with $G_i$ as orbits.
Thus $\{R(g)\ |\ g\in G\}$ is a subgroup of $\Aut(\Gamma)$ isomorphic to $G$.

To end this section, we give a result about the elementary abelian $2$-groups admitting O$m$SRs of valency two, see the Remark~3.5 in \cite{DuFengSpiga}.
\begin{prop}\label{prop=OmSR} $\mz_1$ admits an O$m$SR of valency two if and only if $m\geq 7$, and $\mz_2^t$ admits an O$m$SR of valency two if and only if $m\geq 3$ with $1\leq t\leq 2$.
\end{prop}
%

\section{Proof of Theorem~\ref{theo=main2} }

This section aims to give the proof of Theorem~\ref{theo=main2}, which will be divided into two lemmas. In our first lemma we deal with cyclic groups.

\begin{lem}\label{lem=cyclic}
Let $G$ be a cyclic group, and let $m\geq2$ be an integer. Then
$G$ has no O$m$SR of valency two if and only if
\begin{itemize}
  \item [\rm (1)] $m=2$ and $G\cong \mz_1$ or $\mz_2$;
  \item [\rm (2)] $3\leq m\leq 6$ and $G\cong \mz_1$.
\end{itemize}
\end{lem}

\begin{proof}
By Proposition~\ref{prop=OmSR}, $\mz_1$ has O$m$SR of valency two if and only if $m\geq 7$, and $\mz_2$ has O$m$SR of valency two if and only if $m\geq 3$.
Now, we are in the position to consider the cyclic groups of order at least 3.
Assume $G=\langle a\rangle$, where the order $o(a)\geq 3$.
Take $T_{i,j}\subset G$ with $i,j \in \mz_m$ as follows:
\begin{eqnarray*}
T_{0,0}=\{a\},\,\,\,\,\,\,\,\,\,\,\,\,T_{i,i}&=&\{a^{-1}\}\,\,\,\,\,\,\,\,\,\,\,\,\,\,\,\mbox{ for } i\neq 0; \\
T_{m-1,0}=\{a\},\,\,\,\,\,\,T_{i,i+1}&=&\{1\}\,\,\,\,\,\,\,\,\,\,\,\,\,\,\,\,\,\,\,\,\,\mbox{ for } i\not=m-1; \\
T_{i,j}&=&\emptyset\,\,\,\,\,\,\,\,\,\,\,\,\,\,\,\,\,\,\,\,\,\,\,\,\,\,\, \mbox{ otherwise.}
\end{eqnarray*}
Let $\G_m=\Cay(G,T_{i,j}:i,j \in \mz_m)$ and $A=\Aut(\G_m)$. Then $\G_m$ is an oriented $m$-Cayley digraph over $G$ of valency two.
Figure~\ref{Fig1} might be of some help for understanding the structure of $\G_m$.

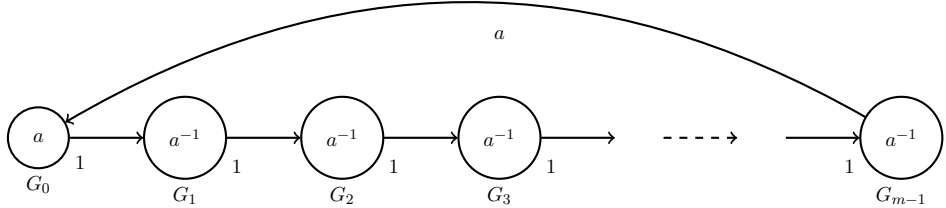
\begin{figure}[!hhh]
\begin{center}
\begin{tikzpicture}[node distance=1.4cm,thick,scale=0.7,every node/.style={transform shape}]
\node[circle](AAAA0){};
\node[left=of AAAA0,circle](AAA0){};
\node[left=of AAA0,circle](AA0){};
\node[right=of AA0,circle,inner sep=9pt, label=45:](A0){};
\node[left=of AA0,circle,inner sep=5pt, label=45:](A1){};
\node[below=of A1,circle,draw,inner sep=9pt,label=-90:$G_1$,label=-20:$1$](A2){$a^{-1}$};
\node[left=of A2,circle,draw, inner sep=9pt, label=-90:$G_0$,label=-20:$1$](A3){$a$};
\node[right=of A2,circle,draw, inner sep=9pt, label=-90:$G_2$,,label=-20:$1$](A4){$a^{-1}$};
\node[right=of A4,circle,draw, inner sep=9pt, label=-90:$G_3$,label=-20:$1$](A5){$a^{-1}$};
\node[above=of A5,circle, inner sep=9pt, label=-90:$a$](A55){};
\node[right=of A5,circle, inner sep=9pt](A6){};
\node[right=of A6,circle, inner sep=9pt](A7){};
\node[right=of A7,circle,draw, inner sep=9pt, label=-90:$G_{m-1}$,label=-160:$1$](A8){$a^{-1}$};
\draw[->](A3) to  (A2);
\draw[->](A2) to  (A4);
\draw[->](A4) to  (A5);
\draw[->](A5) to  (A6);
\draw[->](A7) to  (A8);
\draw[dashed][->] (A6) to (A7);
\draw [->] (A8) to [bend right] (A3);
\end{tikzpicture}
\end{center}
\caption{The oriented $m$-Cayley graphs of valency two for cyclic groups} \label{Fig1}
\end{figure}

It is easy to see that
\begin{eqnarray*}
\G_m^+(1_0)=\{a_0,1_{1}\}, & &  \G_m^+(1_{m-1})=\{a^{-1}_{m-1},a_0\},\\ \G_m^+(1_i)=\{a^{-1}_i,1_{i+1}\} & & \mbox{ for } i\not=0,m-1.
\end{eqnarray*}
Let $\G_m^+(\G_m^+(1_i))$
be the set of out-neighbors of $\G_m^+(1_i)$.
Then, we have
\begin{eqnarray*}
\G_2^+(\G_2^+(1_{0}))&=&\G_2^+(a_{0})\cup\G_2^+(1_1)=\{a^2_{0},a_1\}\cup \{a^{-1}_1,a_0\}=\{a_0,a^2_{0},a_1,a^{-1}_1\};\\
\G_m^+(\G_m^+(1_{0}))&=&\G_m^+(a_{0})\cup\G_m^+(1_1)=\{a^2_{0},a_1\}\cup \{a^{-1}_1,1_2\}=\{a^2_{0},a_1,a^{-1}_1,1_2\} \mbox{ for } m\geq 3;\\
\G_m^+(\G_m^+(1_{m-2}))&=&\G_m^+(a^{-1}_{m-2})\cup\G_m^+(1_{m-1})=
\{a^{-2}_{m-2},a^{-1}_{m-1}\}\cup \{a^{-1}_{m-1},a_0\}= \{a^{-2}_{m-2},a^{-1}_{m-1},a_0\};\\
\G_m^+(\G_m^+(1_{m-1}))&=&\G_m^+(a^{-1}_{m-1})\cup\G_m^+(a_{0})=\{a^{-2}_{m-1},1_0\}\cup \{a^2_0,a_1\}=\{a^{-2}_{m-1},1_0,a^2_0,a_1\};\\
\G_m^+(\G_m^+(1_i))&=&\G_m^+(a^{-1}_i)\cup\G_m^+(1_{i+1})=\{a^{-2}_i,a^{-1}_{i+1}\}\cup \{a^{-1}_{i+1},1_{i+2}\}=\{a^{-2}_i,a^{-1}_{i+1},1_{i+2}\} \mbox{  when $i \neq 0, m-2, m-1$.}
\end{eqnarray*}
Recall that $o(a)\geq3$. Then $|\G_m^+(\G_m^+(1_{0}))|=4$ and $|\G_m^+(\G_m^+(1_{i}))|=3$ for $1\leq i\leq m-2$. Moreover, $|\G_m^+(\G_m^+(1_{m-1}))|=|\{a^{-2}_{m-1},1_0,a^2_0,a_1\}|=3$ when $m=2$ and $o(a)=3$; and $|\G_m^+(\G_m^+(1_{m-1}))|=4$ otherwise. Now, we consider the cases $m=2$ and $m\geq 3$ separately.
Let $m=2$. If $o(a)=3$ or $4$, by Magma~\cite{magma}, $\G_2$ is an O$2$SR over $G$.
Assume $o(a)\geq 5$. We consider the number of arcs in the induced
sub-digraphs $[\G_2^+(\G_2^+(1_{0}))]$ and $[\G_2^+(\G_2^+(1_{1}))]$, respectively.
Note that $\G_2^+(\G_2^+(1_{0}))=\{a_0,a^2_{0},a_1,a^{-1}_1\}$.
By the definition of $\G_2$, we have
\begin{eqnarray*}
\G_2^+(a_0)=\{a^2_{0},a_1\}, & & \G_2^+(a^2_0)=\{a^3_0,a^2_1\}, \\
\G_2^+(a_1)=\{1_1,a^2_0\}, & & \G_2^+(a^{-1}_1)=\{a^{-2}_1,1_0\}.
\end{eqnarray*}
Since $o(a)\geq 5$, there are exactly three arcs in the induced sub-digraph $[\G_2^+(\G_2^+(1_{0}))]$, that is, $(a_0,a^2_0)$, $(a_0,a_1)$ and $(a_1,a^2_0)$.
On the other hand, since $\G_2^+(\G_2^+(1_{1}))=\{1_0,a^2_{0},a_1,a^{-2}_1\}$ and
\begin{eqnarray*}
\G_2^+(1_0)=\{a_{0},1_1\}, & & \G_2^+(a^2_0)=\{a^3_0,a^2_1\},\\
\G_2^+(a_1)=\{1_1,a^2_0\},  & & \G_2^+(a^{-2}_1)=\{a^{-3}_1,a^{-1}_0\},
\end{eqnarray*}
the induced sub-digraph $[\G_2^+(\G_2^+(1_{1}))]$ has only one arc, that is, $(a_1,a^2_0)$. Thus, $[\G_2^+(\G_2^+(1_{0}))]\ncong [\G_2^+(\G_2^+(1_{1}))]$, and so $A$ fixes $G_0$ and $G_1$ setwise, respectively. Therefore, $A_{1_0}$ and $A_{1_1}$ fix $\G_2^+(1_0)=\{a_0,1_1\}$ and $\G_2^+(1_1)=\{a^{-1}_1,a_0\}$
pointwise, respectively. We then conclude that $A_{1_0}=A_{1_1}=1$ as $\Gamma$ is connected. Since $R(G)\leq A$ is transitive on $G_i$ with $i\in \mz_2$, $A$ has two orbits,
and by the Frattini argument, we have $A=R(G)A_{1_0}=R(G)$. Hence $\G_2$ is an O$2$SR over $G$ of valency two.

Let $m\geq 3$. Then $|\G_m^+(\G_m^+(1_{m-1}))|=4$, and so both $G_0\cup G_{m-1}$ and $\Delta=\{G_i~|~1\leq i\leq m-2\}$ are fixed setwise by $A$. Since $\Delta$ has out-neighbors in $G_{m-1}$ while has no out-neighbor in $G_0$ (see Figure~\ref{Fig1}), we have that $A$ fixes $G_0$ and $G_{m-1}$
setwise. Recall that $T_{i,i+1}=\{1\}$ and $T_{i,j}=\emptyset$ for each $1\leq i\leq m-2$ and $j\neq i,i+1$. By considering the out-neighbors of $G_i$ in turn (see Figure~\ref{Fig1}), we notice that for each $0\leq i\leq m-1$ (write $i+1=0$ when $i=m-1$), $G_i$
has out-neighbors only in $G_i$ and $G_{i+1}$, and so it is fixed by $A$ setwise. On the other hand, since $$\G_m^+(1_0)=\{a_0,1_1\},\ \G_m^+(1_{m-1})=\{a^{-1}_{m-1},a_{0}\},\ \G_m^+(1_i)=\{a^{-1}_i,1_{i+1}\},~i\neq 0,m-1,$$
the stabilizer $A_{1_i}$ fixes $\G_m^+(1_i)$ pointwise with $i\in\mz_m$, and so $A_{1_0}=1$ as $\G$ is connected.
Again by Frattini argument, $A=R(G)A_{1_0}=R(G)$, and hence $\G_m$ is an O$m$SR over $G$ of valency two. The proof is completed.
\end{proof}

Finally, we deal with non-cyclic groups generated by two elements.

\begin{lem}\label{lem=twogenerators}
Let $G$ be a non-cyclic group generated by two elements, and let $m\geq2$ be an integer. Then $G$ has an O$m$SR of out-valency two except $(m,G)=(2,\mz_2^2)$.
\end{lem}

\begin{proof}
Let $G=\langle a,b\rangle$ with $o(a)\geq o(b)\geq 2$. Then $G=\langle a,ab\rangle$.
If all of $a$, $b$ and $ab$ are involutions, then $G=\mz_2^2$,
and it has an O$m$SR of valency two if and only if $m\geq3$ by Proposition~\ref{prop=OmSR}. In the following, we assume that at least one of $a$, $b$ and $ab$ has order at least 3. Without loss of generality, we assume $o(a)\geq 3$. In the following, we divide the proof into two cases.

\medskip
{\f\bf Case 1:} $G$ is abelian.
\medskip

Take $T_{i,j}\subset G$ with $i,j \in \mz_m$ as follows:
\begin{eqnarray*}
T_{0,0}=\{a\},\,\,\,\,\,\,\,\,\,\,\,\,T_{i,i}&=&\{ab\}\,\,\,\,\,\,\,\,\,\,\,\,\,\,\,\,\,\,\mbox{ for } i\neq 0; \\ T_{m-1,0}=\{b\},\,\,\,\,\,\, T_{i,i+1}&=&\{1\}\,\,\,\,\,\,\,\,\,\,\,\,\,\,\,\,\,\,\,\,\,\mbox{ for } i\not=m-1; \\
T_{i,j}&=&\emptyset\,\,\,\,\,\,\,\,\,\,\,\,\,\,\,\,\,\,\,\,\,\,\,\,\,\,\, \mbox{ otherwise.}
\end{eqnarray*}
Let $\G_m=\Cay(G,T_{i,j}:i,j \in \mz_m)$ and $A=\Aut(\G_m)$. Then $\G_m$ is an oriented $m$-Cayley digraph over $G$ of valency two.
Figure~\ref{Fig2} might be of some help for understanding the structure of $\G_m$.

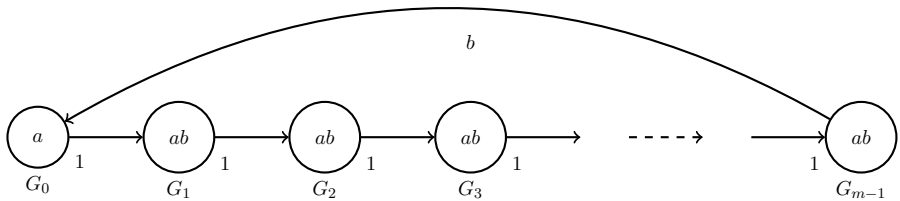
\begin{figure}[!hhh]
\begin{center}
\begin{tikzpicture}[node distance=1.4cm,thick,scale=0.7,every node/.style={transform shape}]
\node[circle](AAAA0){};
\node[left=of AAAA0,circle](AAA0){};
\node[left=of AAA0,circle](AA0){};
\node[right=of AA0,circle,inner sep=9pt, label=45:](A0){};
\node[left=of AA0,circle,inner sep=5pt, label=45:](A1){};
\node[below=of A1,circle,draw,inner sep=9pt,label=-90:$G_1$,label=-20:$1$](A2){$ab$};
\node[left=of A2,circle,draw, inner sep=9pt, label=-90:$G_0$,label=-20:$1$](A3){$a$};
\node[right=of A2,circle,draw, inner sep=9pt, label=-90:$G_2$,,label=-20:$1$](A4){$ab$};
\node[right=of A4,circle,draw, inner sep=9pt, label=-90:$G_3$,label=-20:$1$](A5){$ab$};
\node[above=of A5,circle, inner sep=9pt, label=-90:$b$](A55){};
\node[right=of A5,circle, inner sep=9pt](A6){};
\node[right=of A6,circle, inner sep=9pt](A7){};
\node[right=of A7,circle,draw, inner sep=9pt, label=-90:$G_{m-1}$,label=-160:$1$](A8){$ab$};
\draw[->](A3) to  (A2);
\draw[->](A2) to  (A4);
\draw[->](A4) to  (A5);
\draw[->](A5) to  (A6);
\draw[->](A7) to  (A8);
\draw[dashed][->] (A6) to (A7);
\draw [->] (A8) to [bend right] (A3);
\end{tikzpicture}
\end{center}
\caption{The oriented $m$-Cayley graphs of valency two for abelian groups} \label{Fig2}
\end{figure}

Note that
\begin{eqnarray*}
\G_m^+(1_0)=\{a_0,1_{1}\},\,\,\,\,\,\,\,\,\,\, & &  \,\,\G_m^+(1_{m-1})=\{(ab)_{m-1},b_0\},\\ \G_m^+(1_i)=\{(ab)_i,1_{i+1}\} & & \mbox{ for } i\not=0,m-1.
\end{eqnarray*}
Similar as the proof of Lemma~\ref{lem=cyclic}, we consider the set
$\G_m^+(\G_m^+(1_i))$ for each $i\in\mz_m$.
By the definition of $\G_m$, we have
\begin{eqnarray*}
\G_2^+(\G_2^+(1_{0}))&=&\G_2^+(a_{0})\cup\G_2^+(1_1)=\{a^2_{0},a_1,(ab)_1,b_0\};\\
\G_m^+(\G_m^+(1_{0}))&=&\G_m^+(a_{0})\cup\G_m^+(1_1)=\{a^2_{0},a_1,(ab)_1,1_2\} \mbox{ for } m\geq 3;\\
\G_m^+(\G_m^+(1_{m-2}))&=&\G_m^+((ab)_{m-2})\cup\G_m^+(1_{m-1})=\{(ab)^2_{m-2},(ab)_{m-1},b_0\};\\
\G_m^+(\G_m^+(1_{m-1}))&=&\G_m^+((ab)_{m-1})\cup\G_m^+(b_{0})=\{(ab)^2_{m-1},(ab^2)_0,(ab)_0,b_1\};\\
\G_m^+(\G_m^+(1_i))&=&\G_m^+((ab)_i)\cup\G_m^+(1_{i+1})=\{(ab)^2_i,(ab)_{i+1},1_{i+2}\} \mbox{ when $i \neq 0, m-2, m-1$.}
\end{eqnarray*}
Since $o(b)\geq 2$, we have $ab^2\neq ab$, and so  $|\G_m^+(\G_m^+(1_{0}))|=|\G_m^+(\G_m^+(1_{m-1}))|=4$ and
$|\G_m^+(\G_m^+(1_{i}))|=3$ for $i\neq0,m-1$. It follows that
$A$ fixes $G_0\cup G_{m-1}$ setwise.

First, assume $m=2$. Since $G=\langle a,b\rangle$, we have that $\G_m$
is a connected digraph. Let $\G'(1_i)$ be the set of vertices of distance at most 2 from $1_i$, that is, $\G'(1_i)=\{1_i\}\cup \G_2^+(1_i)\cup \G_2^+(\G_2^+(1_i))$ with $i\in\mz_2$.
We consider the induced sub-digraphs $[\G'(1_0)]$ and $[\G'(1_1)]$, respectively.
Note that $\G_2^+(\G_2^+(1_{0}))=\{a^2_{0},a_1,(ab)_1,b_0\}$ and $\G_2^+(\G_2^+(1_{1}))=\{(ab)^2_{1},(ab^2)_0,(ab)_0,b_1\}$.
By the definition of $\G_2$, we have
\begin{eqnarray*}
\G_2^+(a^2_0)=\{a^3_0,a^2_1\},& &  \G_2^+(a_1)=\{(a^2b)_1,(ba)_0\}, \\
\G_2^+((ab)_1)=\{(ab)^2_1,(ab^2)_0\},& & \G_2^+(b_0)=\{(ab)_{0},b_1\}, \\
\G_2^+((ab)^2_1)=\{(ab)^3_1,(a^2b^3)_0\}, & & \G_2^+((ab^2)_0)=\{(ab)^2_{0},(ab^2)_1\},\\
\G_2^+((ab)_0)=\{(a^2b)_0,(ab)_1\},  & & \G_2^+(b_1)=\{(a^2b)_1,b^2_0\}.
\end{eqnarray*}
One may see Figure~\ref{Fig3} for the induced sub-digraphs $[\G'(1_0)]$ and $[\G'(1_1)]$,
and obviously, $[\G'(1_0))]\ncong [\G'(1_1)]$.
It yields that $A$ fixes $G_0$ and $G_1$ setwise, respectively.
Furthermore, the stabilizer $A_{1_0}$ fixes $\G_2^+(1_0)=\{a_0,1_1\}$ pointwise, while $A_{1_1}$ fixes $\G_2^+(1_1)=\{(ab)_1,b_0\}$ pointwise.
The connectivity of $\G$ forces $A_{1_0}=A_{1_1}=1$.
Since $R(G)\leq A$ is transitive on $G_i$ with $i\in \mz_2$, we have $A=R(G)A_{1_0}=R(G)$, and hence $\G_2$ is an O$2$SR over $G$ of valency two.

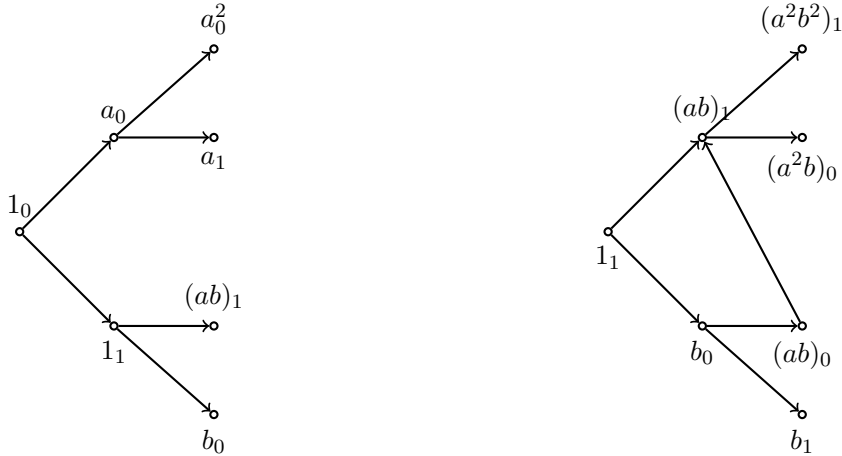
\begin{figure}[!ht]
\begin{tikzpicture}[node distance=1.1cm,thick,scale=0.6,every node/.style={transform shape},scale=1.6]
\node[circle](A01){};
\node[right=of A01, circle,draw,inner sep=1pt, label=above:{$1_0$}](A0){};
\node[right=of A0](B0){};
\node[below=of B0, circle,draw, inner sep=1pt, label=below:{$1_1$}](B2){};
\node[above=of B0, circle,draw, inner sep=1pt, label=above:{$a_0$}](B1){};
\node[right=of B0](C0){};
\node[below=of C0, circle,draw, inner sep=1pt, label=above:{$(ab)_1$}](C3){};
\node[below=of C3, circle,draw, inner sep=1pt, label=below:{$b_0$}](C4){};
\node[above=of C0, circle,draw, inner sep=1pt, label=below:{$a_1$}](C1){};
\node[above=of C1, circle,draw, inner sep=1pt, label=above:{$a^2_0$}](C2){};

\draw[->] (A0) to (B1);
\draw[->] (A0) to (B2);
\draw[->] (B1) to (C1);
\draw[->] (B1) to (C2);
\draw[->] (B2) to (C3);
\draw[->] (B2) to (C4);

\node[right=of C0](C02){};
\node[right=of C02](C03){};
\node[right=of C03](C01){};
\node[right=of C01, circle,draw, inner sep=1pt, label=below:{$1_1$}](A01){};
\node[right=of A01](B01){};
\node[below=of B01, circle,draw, inner sep=1pt, label=below:{$b_0$}](B21){};
\node[above=of B01, circle,draw, inner sep=1pt, label=above:{$(ab)_1$}](B11){};
\node[right=of B01](C01){};
\node[below=of C01, circle,draw, inner sep=1pt, label=below:{$(ab)_0$}](C31){};
\node[below=of C31, circle,draw, inner sep=1pt, label=below:{$b_1$}](C41){};
\node[above=of C01, circle,draw, inner sep=1pt, label=below:{$(a^2b)_0$}](C11){};
\node[above=of C11, circle,draw, inner sep=1pt, label=above:{$(a^2b^2)_1$}](C21){};

\draw[->] (A01) to (B11);
\draw[->] (A01) to (B21);
\draw[->] (B11) to (C11);
\draw[->] (B11) to (C21);
\draw[->] (B21) to (C31);
\draw[->] (B21) to (C41);
\draw[->] (C31) to (B11);
\end{tikzpicture}
\caption{The induced sub-digraphs $[\G'(1_0)]$ and $[\G'(1_1)]$}\label{Fig3}
\end{figure}

Next, let $m\geq 3$. Then $A$ fixes both $G_0\cup G_{m-1}$ and $\Delta=\{G_i~|~1\leq i\leq m-2\}$ setwise. Since $\Delta$ has out-neighbors in
$G_{m-1}$ but no out-neighbor in $G_0$, we conclude that $A$ fixes $G_0$ and $G_{m-1}$ setwise. Recalling that $T_{i,i+1}=\{1\}$ and $T_{i,j}=\emptyset$ for each $1\leq i\leq m-2$ and $j\neq i,i+1$ (see Figure~\ref{Fig2}), we can conclude that $A$ fixes $G_i$ setwise for each $1\leq i\leq m-2$. On the other hand, since
$$\G_m^+(1_{0})=\{a_{0},1_1\},\ \G_m^+(1_{m-1})=\{(ab)_{m-1},b_{0}\},\ \G_m^+(1_i)=\{(ab)_i,1_{i+1}\},~i\neq m-1,$$ we have $A_{1_i}$ fixes $\G_m^+(1_i)$ pointwise, and so $A_{1_0}=A_{1_1}=1$ as $\G$ is connected.
Since $R(G)\leq A$ is transitive on $G_i$ with $i\in \mz_m$, we conclude that $A=R(G)A_{1_0}=R(G)$ and $\G_m$ is an O$m$SR over $G$ of valency two.

\medskip
{\f\bf Case 2:} $G$ is non-abelian.
\medskip

Take $T_{i,j}\subset G$ with $i,j \in \mz_m$ as follows:
\begin{eqnarray*}
T_{i,i}&=&\{a\}\,\,\,\,\,\,\,\,\,\,\,\,\,\,\,\,\,\,\,\,\,\mbox{ for } i\in\mz_m; \\ T_{m-1,0}=\{b\},\,\,\,\,\,\, T_{i,i+1}&=&\{1\}\,\,\,\,\,\,\,\,\,\,\,\,\,\,\,\,\,\,\,\,\,\mbox{ for } i\not=m-1; \\
T_{i,j}&=&\emptyset\,\,\,\,\,\,\,\,\,\,\,\,\,\,\,\,\,\,\,\,\,\,\,\,\,\,\, \mbox{ otherwise.}
\end{eqnarray*}
Let $\G_m=\Cay(G,T_{i,j}:i,j \in \mz_m)$ and $A=\Aut(\G_m)$. Then $\G_m$ is an oriented $m$-Cayley digraph over $G$ of valency two.
Figure~\ref{Fig4} might be of some help for understanding the structure of $\G_m$.

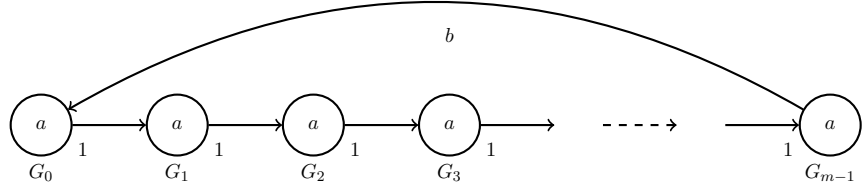
\begin{figure}[!hhh]
\begin{center}
\begin{tikzpicture}[node distance=1.4cm,thick,scale=0.7,every node/.style={transform shape}]
\node[circle](AAAA0){};
\node[left=of AAAA0,circle](AAA0){};
\node[left=of AAA0,circle](AA0){};
\node[right=of AA0,circle,inner sep=9pt, label=45:](A0){};
\node[left=of AA0,circle,inner sep=5pt, label=45:](A1){};
\node[below=of A1,circle,draw,inner sep=9pt,label=-90:$G_1$,label=-20:$1$](A2){$a$};
\node[left=of A2,circle,draw, inner sep=9pt, label=-90:$G_0$,label=-20:$1$](A3){$a$};
\node[right=of A2,circle,draw, inner sep=9pt, label=-90:$G_2$,,label=-20:$1$](A4){$a$};
\node[right=of A4,circle,draw, inner sep=9pt, label=-90:$G_3$,,label=-20:$1$](A5){$a$};
\node[above=of A5,circle, inner sep=9pt, label=-90:$b$](A55){};
\node[right=of A5,circle, inner sep=9pt](A6){};
\node[right=of A6,circle, inner sep=9pt](A7){};
\node[right=of A7,circle,draw, inner sep=9pt, label=-90:$G_{m-1}$,label=-160:$1$](A8){$a$};
\draw[->](A3) to  (A2);
\draw[->](A2) to  (A4);
\draw[->](A4) to  (A5);
\draw[->](A5) to  (A6);
\draw[->](A7) to  (A8);
\draw[dashed][->] (A6) to (A7);
\draw [->] (A8) to [bend right] (A3);
\end{tikzpicture}
\end{center}
\caption{The oriented $m$-Cayley graphs of valency two for non-abelain groups generated by two elements} \label{Fig4}
\end{figure}

Similarly, we consider $\G_m^+(\G_m^+(1_i))$, the set of out-neighbors of $\G_m^+(1_i)$.
Since
\begin{eqnarray*}
\G_m^+(1_{m-1})=\{a_{m-1},b_0\}, & & \G_m^+(1_i)=\{a_i,1_{i+1}\}\,\,\,\,\,\, \mbox{ for } i\not=m-1,
\end{eqnarray*}
we have
\begin{eqnarray*}
\G_m^+(\G_m^+(1_{m-2}))&=&\G_m^+(a_{m-2})\cup\G_m^+(1_{m-1})=\{a^2_{m-2},a_{m-1}\}\cup \{a_{m-1},b_0\}=\{a^2_{m-2},a_{m-1},b_0\}, \\
\G_m^+(\G_m^+(1_{m-1}))&=&\G_m^+(a_{m-1})\cup\G_m^+(b_{0})=\{a^2_{m-1},(ba)_0\}\cup \{(ab)_0,b_1\}=\{a^2_{m-1},(ab)_{0},(ba)_0,b_1\}, \\
\G_m^+(\G_m^+(1_i))&=&\G_m^+(a_i)\cup\G_m^+(1_{i+1})=\{a^2_i,a_{i+1}\}\cup \{a_{i+1},1_{i+2}\}=\{a^2_i,a_{i+1},1_{i+2}\} \,\,\,\mbox{when $i \neq 0, m-2, m-1$.}
\end{eqnarray*}
Since $(ab)_0\neq (ba)_0$ because $G$ is non-abelian, $|\G_m^+(\G_m^+(1_{m-1}))|=4$ and $|\G_m^+(\G_m^+(1_i))|=3$ for each $i\neq m-1$.
Therefore, $A$ fixes $G_{m-1}$ setwise.
Since $G_i$ only has out-neighbors in $G_{i+1}$ for each $0\leq i\leq m-1$ (write $i+1=0$ when $i=m-1$), it is fixed by $A$ setwise.
Recall that $\G_m^+(1_{m-1})=\{a_{m-1},b_{0}\}$ and $\G_m^+(1_i)=\{a_i,1_{i+1}\}$ for each $i\neq m-1$.
Thus, $A_{1_i}$ fixes $\G_m^+(1_i)$ pointwise, and so $A_{1_0}=1$ as $\G$ is connected.
Since $R(G)\leq A$ is transitive on $G_i$ with $i\in \mz_m$, we have $A=R(G)A_{1_0}=R(G)$ and hence $\G_m$ is an O$m$SR over $G$ of valency two. The proof is completed.
\end{proof}

\f {\bf Acknowledgement:} The first author was supported by the National Natural Science Foundation of China (12101601,12161141005), the Natural Science Foundation of Jiangsu Province, China (BK20200627) and by the China Postdoctoral Science Foundation (2021M693423). The second author was supported by the Basic Science Research Program through the National Research Foundation of Korea (NRF) funded by the Ministry of Education (2018R1D1A1B05048450). The third author was supported by the National Natural Science Foundation of China (12101070,1201101021,11731002,12161141005).

\end{document}